\documentclass[11pt,a4paper]{article}%
\usepackage{amssymb}
\usepackage{amsmath}
\usepackage{graphicx}
\usepackage[ansinew]{inputenc}
\usepackage{amsfonts}%
\providecommand{\U}[1]{\protect\rule{.1in}{.1in}}
\newtheorem{theorem}{Theorem}

\newtheorem{corollary}[theorem]{Corollary}

\newtheorem{definition}[theorem]{Definition}

\newtheorem{proposition}[theorem]{Proposition}

\newenvironment{proof}[1][Proof]{\textbf{#1.} }{\ \rule{0.5em}{0.5em}}
\ifx\pdfoutput\relax\let\pdfoutput=\undefined\fi
\newcount\msipdfoutput
\ifx\pdfoutput\undefined\else
\ifcase\pdfoutput\else
\ifx\paperwidth\undefined\else
\ifdim\paperheight=0pt\relax\else\pdfpageheight\paperheight\fi
\ifdim\paperwidth=0pt\relax\else\pdfpagewidth\paperwidth\fi
\fi\fi\fi
\begin{document}

\title{Climbing on Pyramids}
\author{\textbf{Jean Serra and Bangalore\ Ravi Kiran}\\Université Paris-Est \\Laboratoire d'Informatique Gaspard-Monge\\A3SI, ESIEE Paris, 2 Bd Blaise Pascal, B.P. 99 \\93162 Noisy-le-Grand CEDEX, France }
\date{15, march 2012}
\maketitle

\begin{abstract}
A new approach is proposed for finding the "best cut" in a hierarchy of
partitions by energy minimization. Said energy must be "climbing" i.e. it must
be hierarchically and scale increasing. It encompasses separable energies and
those composed under supremum.

\end{abstract}

\section{Introduction}

The present note\footnote{This work received funding from the Agence Nationale
de la Recherche through contract ANR-2010-BLAN-0205-03 KIDIKO.} extends the
results of \cite{SER11a}, which themselves generalize some results of
L.\ Guigues' Phd thesis \cite{GUI03} (see also \cite{GUI06}). In \cite{GUI03},
any partition, or partial partition, of the space is associated with a
"separable energy", i.e. with an energy whose value for the whole partition is
the sum of the values over its various classes. Under this assumption, two
problems are treated:

\begin{enumerate}
\item given a hierarchy of partitions and a separable energy $\omega$, how to
combine some classes of the hierarchy in order to obtain a new partition that
minimizes $\omega$?

\item when $\omega$ depends on integer $j$, i.e. $\omega=\omega^{j}$, how to
generate a sequence of minimum partitions that is increasing in $j$, which
therefore should form a minimum hierarchy?
\end{enumerate}

Though\ L.\ Guigues exploited linearity and affinity assumptions to lean
original models on, it is not sure that they are the very cause of the
properties he found. Indeed, for solving problem 1 above, an alternative and
simpler condition of increasingness is proposed in \cite{SER11a}. After
additional developments, it leads to the theorem 4 of this paper. The second
question, of a minimum hierarchy, which was not treated in \cite{SER11a}, is
the concern of sections 3 to 5 of the text. The main results are the theorem
\ref{pyram_croit}, and the new algorithm developed in Section 5, more general
but simpler than that of \cite{GUI03}. It is followed by the two sections 6
and 7 about the additive and the $\vee$-composed energies respectively. They
show that Theorem \ref{pyram_croit} applies to linear energies (e.g. Mumford
and Shah, Salembier and Garrido), as well to several useful non linear
energies (Soille, Zanoguera).

\section{Hierarchy of partitions (reminder)}%

\begin{figure}
[ptb]
\begin{center}
\ifcase\msipdfoutput
\includegraphics[
height=2.2736in,
width=5.047in
]%
{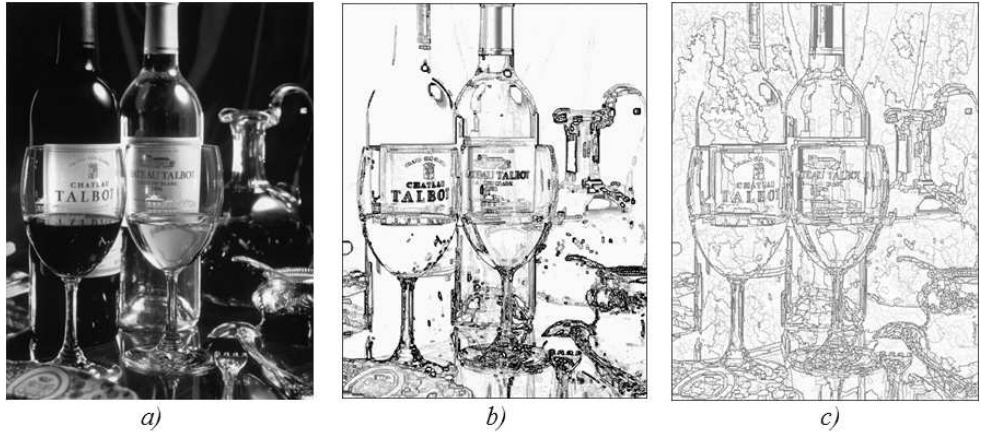}%
\else
\includegraphics[
height=2.2736in,
width=5.047in
]%
{C:/SWtexts/2012/climbing_on_pyramids_2012/graphics/Talbot_N-B__1.pdf}%
\fi
\caption{\textit{a}) Basic input, \textit{b}) and \textit{c}) saliency maps from increasing watershed w.r.t. floodings by dynamics [\textit{b})] and by volume [\textit{c}%
)\cite{MEY09a}, \cite{COU08}.}
\label{talbot}%
\end{center}
\end{figure}

The space under study (Euclidean, digital, or else) is denoted by $E$, and the
set of all partitions of $E$ by $\mathcal{D}_{0}(E)$. Here a convenient notion
is that \textit{partial partition, }of C.\ Ronse\textit{\ }\cite{RON08}%
\textit{. }When we associate a partition $\pi(A)$ of a set $A\in
\mathcal{P}(E)$ and nothing outside $A$, then $\pi(A)$ is called a
\textit{partial partition} of $E$ of support $A$. The family of all partial
partitions of set $E$ is denoted by $\mathcal{D}(E)$, or simply by
$\mathcal{D}$ when there is no ambiguity.

\bigskip

Finite hierarchies of partitions appeared initially in taxonomy, for
\ classifying objects. One can quote in particular the works of J.P.\ Benzécri
\cite{BZCRI84} and of E. Diday \cite{DID82}. We owe to the first author the
theorem linking ultrametrics with hierarchies, namely the equivalence between
statements 1 and 3 in theorem \ref{3equivalences} below. Hierarchies $H$ of
partitions usually derive from a chain of segmentations of some given function
$f$ on set $E$, i.e. from a stack of scalar or vector images, a chain which
then serves as the framework for further operators. We consider, here,
function $f$ and hierarchy $H$ as two starting points, possibly independent.
This results in the following definition:

\begin{definition}
Let $\mathcal{D}_{0}(E)$ be the set of all partitions of $E$, equipped with
the refinement ordering. A hierarchy $H$, of partitions $\pi_{i}$ of $E$ is a
finite chain in\ $\mathcal{D}_{0}(E)$, i.e.
\begin{equation}
H=\{\pi_{i},0\leq i\leq n,\pi_{i}\in\mathcal{D}_{0}(E)\mid i\leq k\leq
n\Rightarrow\pi_{i}\leq\pi_{k}\}, \label{hierarchy 1}%
\end{equation}
of extremities the universal extrema of $\mathcal{D}_{0}(E)$, namely $\pi
_{0}=\{\{x\},x\in E\}$ and $\pi_{n}=E$.
\end{definition}

Let $S_{i}(x)$ be the class of partition $\pi_{i}$ of $H$ at point $x\in E$.
Denote by $\mathcal{S}$ the set of all classes $S_{i}(x)$ , i.e.
$\mathcal{S}=\{S_{i}(x),x\in E,0\leq i\leq n\}$. Expression (\ref{hierarchy 1}%
) means that at each point $x\in E$ the family of those classes $S_{i}(x)$ of
$\mathcal{S}$ that contain $x$ forms a finite chain $\mathcal{S}_{x}$ in
$\mathcal{P}(E)$, of nested elements from $\{x\}$ to $E:$%
\[
\mathcal{S}_{x}=\{S_{i}(x),0\leq i\leq n\}.
\]
%

\begin{figure}
[ptb]
\begin{center}
\ifcase\msipdfoutput
\includegraphics[
height=1.2583in,
width=5.1465in
]%
{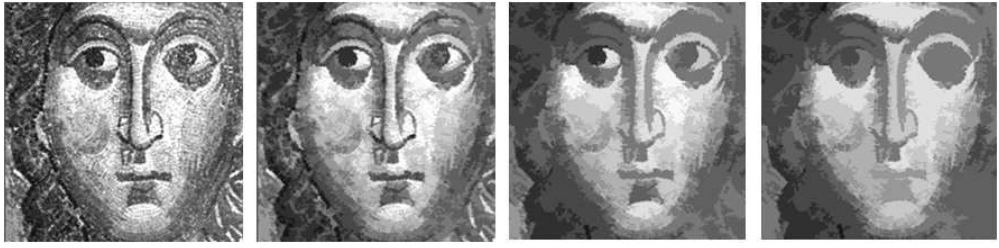}%
\else
\includegraphics[
height=1.2583in,
width=5.1465in
]%
{C:/SWtexts/2012/climbing_on_pyramids_2012/graphics/mosaic__2.pdf}%
\fi
\caption{The initial image has been transformed by increasing alternated
connected filters. They result in partitions into flat zones that increase
from left to right.}%
\label{mosaic}%
\end{center}
\end{figure}

According to a classical result, a family $\{S_{i}(x),x\in E,0\leq i\leq n\}$
of indexed sets generates the classes of a hierarchy iff
\begin{equation}
i\leq j\text{ \ and \ }x,y\in E\text{ \ \ }\Rightarrow\text{ \ \ }%
S_{i}(x)\subseteq S_{j}(y)\text{ \ or \ }S_{i}(x)\supseteq S_{j}(y)\text{ \ or
}S_{i}(x)\cap S_{j}(y)=\varnothing. \label{hierarchy 2}%
\end{equation}

A hierarchy may be represented in space $E$ by the saliencies of its
frontiers, as depicted in Figure \ref{talbot}. Another representation, more
adapted to the present study, emphasizes the classes rather than their edges,
as depicted in Figure\ref{mosaic}. Finally, one can also describe it, in a
more abstract manner, by a family tree where each node of bifurcation is a
class $S$, as depicted in Figure \ref{hierarchy}. The classes of $\pi_{i-1}$
at level $i-1$ which are included in $S_{i}(x)$ are said to be \textit{the
sons} of $S_{i}(x)$. Clearly, the sets of the descenders of each $S$ forms in
turn a hierarchy $H(S)$ of summit $S$, which is included in the complete
hierarchy $H=H(E)$.%

\begin{figure}
[ptb]
\begin{center}
\ifcase\msipdfoutput
\includegraphics[
height=1.6457in,
width=5.1465in
]%
{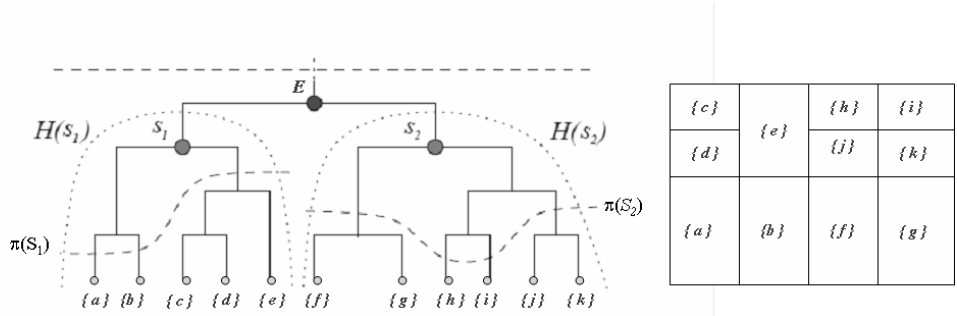}%
\else
\includegraphics[
height=1.6457in,
width=5.1465in
]%
{C:/SWtexts/2012/climbing_on_pyramids_2012/graphics/hierarchy__3.pdf}%
\fi
\caption{Left, hierarchical tree; right, the corresponding space structure.
$S_{1}$ and $S_{2}$ are the nodes sons of $E$, and $H(S_{1})$ and $H(S_{1})$
are the associated sub-hierarchies.$\ \pi_{1}$ and $\pi_{2}$ are cuts of
$H(S_{1})$ and $H(S_{1})$ respectively, and $\pi_{1}\sqcup\pi_{2}$ is a cut of
$E$.}%
\label{hierarchy}%
\end{center}
\end{figure}

The two zones $H(S_{1})$ and $H(S_{2})$, drawn in Figure \ref{hierarchy} in
small dotted lines, are examples of such sub hierarchies. The following
theorem \cite{SER11a} makes more precise the hierarchical structure

\begin{theorem}
\label{3equivalences}The three following statements are equivalent:

\begin{enumerate}
\item $H$ is an indexed hierarchy,

\item the set $\mathcal{S}$ of all classes of the partition $\pi_{i}$ of $H$
forms an ultrametric space, of distance the indexing parameter,

\item every binary criterion\textit{\ }$\sigma:(\mathcal{F},\mathcal{S\cup
\varnothing})\rightarrow\{0,1\}$\textit{\ is connective.}
\end{enumerate}
\end{theorem}

This result shows that the connective segmentation approach is inefficient for
hierarchies, and orients us towards the alternative method, which consists in
optimizing an energy.

\section{Optimum partitioning of a hierarchy}

\subsection{Cuts in a hierarchy}

Following L. Guigues \cite{GUI03} \cite{GUI06}, we say that any partition
$\pi$ of $E$ whose classes are taken in $\mathcal{S}$ defines a \textit{cut}
in hierarchy $H$. The set of all cuts of $E$ is denoted by $\Pi(E)=\Pi$. Every
"horizontal" section $\pi_{i}(H)$ at level $i$ is obviously a cut, but several
levels can cooperate in a same cut, such as $\pi(S_{1})$ and $\pi(S_{2})$,
drawn with thick dotted lines in Figure \ref{hierarchy}. Similarly, the
partition $\pi(S_{1})\sqcup\pi(S_{2})$ generates a cut of $H(E)$. The symbol
$\sqcup$ is used here for expressing that groups of classes are concatenated.
It means that given two partial partitions $\pi(S_{1})$ and $\pi(S_{2})$
having disjoint supports, $\pi(S_{1})\sqcup\pi(S_{2})$ is the partial
partition whose classes are either those of $\pi(S_{1})$ or those of
$\pi(S_{2})$.

Similarly, one can define cuts inside any sub-hierarchy $H(S)$ of summit $S$.
Let $\Pi(S)$ be the family of all cuts of $H(S)$. The union of all these
families, when node $S$ spans hierarchy $H$ is denoted by%
\begin{equation}
\widetilde{\Pi}(H)=\cup\{\Pi(S),S\in\mathcal{S}(H)\}. \label{guigues 3}%
\end{equation}
Although the set $\widetilde{\Pi}(H)$ does not regroup all possible partial
partitions with classes in $\mathcal{S}$, it contains the family $\Pi(E)$ of
all cuts of $H(E)$. The hierarchical structure of the data induces a relation
between the family $\Pi(S)$ of the cuts of node $S$ and the families
$\Pi(T_{1}),..,\Pi(T_{q})$ of the sons $T_{1},..,T_{q}$ of $S$. Since all
expressions of the form $\sqcup\{\pi(T_{k});1\leq k\leq q\}$ define cuts of
$S$, $\Pi(S)$ contains the whole family
\[
\Pi^{\prime}(S)=\{\pi(T_{1})\sqcup..\pi(T_{k})..\sqcup\pi(T_{q}%
);\ \text{\ \ \ \ }\pi(T_{1})\in\Pi(T_{1})...\ \pi(T_{q})\in\Pi(T_{q})\},
\]
plus the cut of $S$ into a unique class, i.e. $S$ itself, which is not a
member of $\Pi^{\prime}(S)$. And as the other unions of several $T_{k}$ are
not classes listed in $\mathcal{S}$, there is no other possible cut, hence
\begin{equation}
\Pi(S)=\Pi^{\prime}(S)\cup S. \label{hier 5}%
\end{equation}

\subsection{Cuts of minimum energy and $h$-increasingness}

In the present context, an energy $\omega:$ $\mathcal{D}(E)\rightarrow
\mathbb{R}^{+\text{ }}$is a non negative numerical function over the family
$\mathcal{D}(E)$ of all partial partitions of set $E$. The cuts of
$\Pi(E)\subseteq\widetilde{\Pi}$ of minimum energy, or \textit{minimum cuts},
are characterized under the assumption of \textit{hierarchical
increasingness,} or more shortly\textit{\ of }$\mathit{h}$%
\textit{-increasingness }\cite{SER11a}.

\begin{definition}
Let $\pi_{1}$ and $\pi_{2}$ be two partial partitions of same support, and
$\pi_{0}$ be a partial partition disjoint from $\pi_{1}$ and $\pi_{2}$. An
energy $\omega$ on $\mathcal{D}(E)$ is said to be \emph{hierarchically
increasing, or }$h$\emph{-increasing},\emph{\ }in $\mathcal{D}(E)$
when,\ $\pi_{0},\pi_{1},\pi_{2}\in\mathcal{D}(E),\ \pi_{0}$ disjoint of
$\pi_{1}$ and $\pi_{2}$, we have
\begin{equation}
\omega(\pi_{1})\leq\omega(\pi_{2})\ \ \Rightarrow\ \ \omega(\pi_{1}\sqcup
\pi_{0})\leq\ \omega(\pi_{2}\sqcup\pi_{0}).\text{ \ } \label{croiss_hierarch}%
\end{equation}

\end{definition}

An illustration of the meaning of implication (\ref{croiss_hierarch}) is given
in Figure \ref{hierach_increas}. When the partial partitions are embedded in a
hierarchy $H$, then Rel.(\ref{croiss_hierarch}) allows us an easy
characterization of the cuts of minimum energy of $H$, according to the
following property, valid for the class $\mathcal{H}$ of all finite
hierarchies on $E$%

\begin{figure}
[ptb]
\begin{center}
\ifcase\msipdfoutput
\includegraphics[
height=1.4001in,
width=5.1084in
]%
{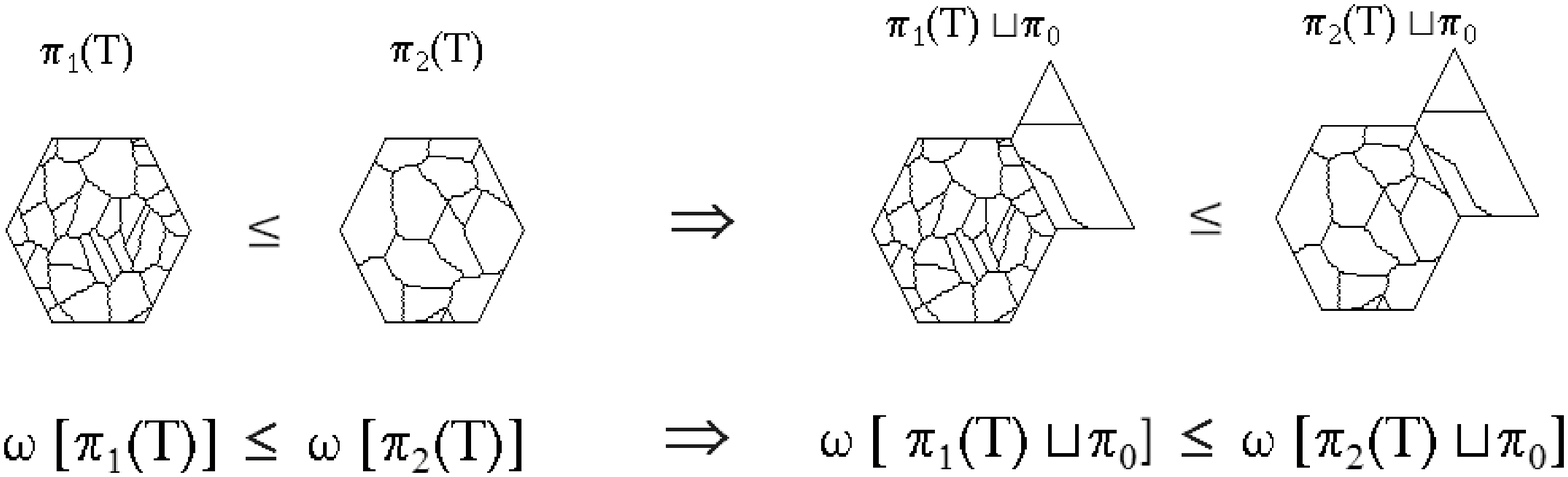}%
\else
\includegraphics[
height=1.4001in,
width=5.1084in
]%
{C:/SWtexts/2012/climbing_on_pyramids_2012/graphics/hierach_increas__4.pdf}%
\fi
\caption{Hierachical increasingness.}%
\label{hierach_increas}%
\end{center}
\end{figure}

\begin{theorem}
\label{coupe_min} Let $H\in\mathcal{H}$ be a finite hierarchy, and $\omega$ be
an energy on $\mathcal{D}(E)$. Consider a node $S$ of $H$ with $p$
sons\ $T_{1}..T_{p}$ of minimum cuts $\pi_{1}^{\ast},..\pi_{p}^{\ast}$. The
cut of minimum energy of node $S$ is either the cut\
\begin{equation}
\pi_{1}^{\ast}\sqcup\pi_{2}^{\ast}..\sqcup\pi_{p}^{\ast}, \label{hier 4}%
\end{equation}
\ or the partition of $\omega$ into a unique class, if and only if $S$ is $h$-increasing.
\end{theorem}

\begin{proof}
We firstly prove that the condition is sufficient. The $h$-increasingness
implies of the energy implies that cut (\ref{hier 4}) has the lowest energy
among all the cuts of type $\Pi^{\prime}(S)=$ $\sqcup\{\pi(T_{k});1\leq k\leq
p\}$ (it does not follow that it is unique). Now, from the decomposition
(\ref{hier 5}), every cut of $S$ is either an element of $\Pi^{\prime}(S)$, or
$S$ itself. Therefore, the set formed by the cut (\ref{hier 4}) and $S$
contains one minimum cut of $S$\ at least.

We will prove that the $h$-increasingness is necessary by means of a
counter-example. Consider the hierarchies of $n$ levels in $E=\mathbb{R}^{2}$,
and associate an energy with the lengths of the frontiers as follows
\begin{align}
\omega(\pi(S))  &  =\omega(T_{1}\sqcup..T_{u}..\sqcup T_{q})=\sum
\limits_{1\leq u\leq q}(\partial T_{u}),\text{ \ \ \ \ \ \ \ \ when
\ \ \ \ \ }\sum\limits_{1\leq u\leq q}(\partial T_{u})\leq5\label{Q}\\
\omega(\pi(S))  &  =\sum\limits_{1\leq u\leq q}(\partial T_{u})-5\text{
\ \ \ \ \ when not.}\nonumber
\end{align}
where $\{T_{u},1\leq u\leq q\}$ are the sons of node $S$. Calculate the
minimum cut of the three levels partition depicted in Fig.\ref{counter_ex1} a)
b) and c). The size of the square is $1$ and one must take half the length for
the external edges. We find $\omega(T_{1}\sqcup T_{2})=3$ and $\omega(S)=2$.
Hence $S$ is the minimum cut of its own sub-hierarchy. So does $S^{\prime}$.
At the next level we have $\omega(S\sqcup S^{\prime})=4 $, and $\omega(E)=3$.
Nevertheless, $E$ is not the minimum cut, because $\omega(T_{1}\sqcup
T_{2}\sqcup T_{1}^{\prime}\sqcup T_{2}^{\prime})=6-5=1! $
\end{proof}

%

\begin{figure}
[ptb]
\begin{center}
\ifcase\msipdfoutput
\includegraphics[
height=1.4209in,
width=5.4544in
]%
{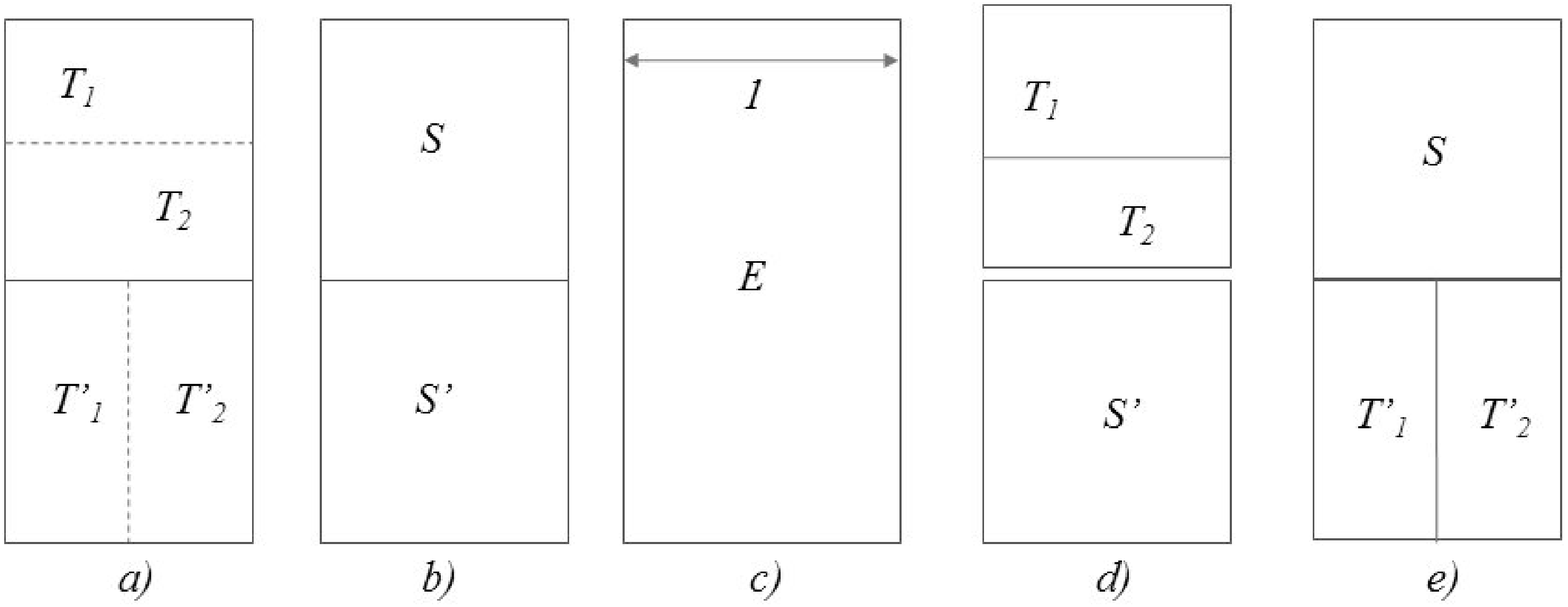}%
\else
\includegraphics[
height=1.4209in,
width=5.4544in
]%
{C:/SWtexts/2012/climbing_on_pyramids_2012/graphics/counter_ex1__5.pdf}%
\fi
\caption{\textit{a}), \textit{b}),\textit{c}): three levels of a hierarchy
$H$. Partition a) is the minimum cut for energy (\ref{Q}), which contradicts
hierarchical increasingness. Partitions \textit{d}) and \textit{e}) are two
minimum cuts of $H$ for the sequence (\ref{QQ}) of energies. They contradict
scale increasingness. \ }%
\label{counter_ex1}%
\end{center}
\end{figure}

The condition of $h$-increasingness (\ref{croiss_hierarch}) opens into a broad
range of energies, and is easy to check. It encompasses the case of the
separable energies \cite{GUI03} \cite{Sal20}, as well as energies composed by
suprema \cite{ANG06a} \cite{SOI08a} \cite{ZAN99}. Computationally, it yields
to the following Guigues'algorithm:

\begin{itemize}
\item \ scan in one pass all nodes of $H$ according to an ascending
lexicographic order ;

\item determine at each node $S$ a temporary minimum cut of $H$ by comparing
the energy of $S$ to that of the concatenation of the temporary minimum cuts
of the (already scanned) sons $T_{k}$ of $S$ .
\end{itemize}

\subsection{Single minimum cuts}

It may happen that in the family $\Pi(S)$ of Relation (\ref{hier 5}) a minimum
cut of $\Pi^{\prime}(S)$ has the same energy as that of $S$. This event
introduces two solutions which are then carried over the whole induction. And
since such a doublet can occur regarding any node $S$ of $H$, the family
$M=M(H)$ of all minimum cuts may be very comprehensive. However, $M$ turns out
to be structured as a complete lattice for the ordering of the refinement,
where the sup (resp. the inf) are obtained by taking the union (resp. the
intersection) of the classes \cite{SER11a}.

\bigskip

The risk of several minimum cuts w.r.t. $\omega$ may becomes sometimes
cumbersome. Then one can always ensure unicity by slightly modifying a
$h$-increasing energy $\omega$. Denote by $m$ the minimum of all the positive
differences of energies involved in the partial partitions of $H$, i.e.%
\[
m=\inf\{\omega(\pi)-\omega(\pi^{\prime}),\omega(\pi)<\omega(\pi^{\prime
})\}\text{ \ \ }\pi,\pi^{\prime}\in\widetilde{\Pi}(H).
\]

As the cardinal of $\widetilde{\Pi}$ is finite, $m$ is strictly positive.
Therefore, one can find a $\varepsilon$ such as $0<\varepsilon<m$, and state
the following:

\begin{proposition}
\label{unicity}Let $\omega$ be a $h$-increasing energy over $\widetilde{\Pi}$.
Introduce the additional energy $\omega^{\prime}$ for all $\{\pi(S)\in\Pi(S),$
$S\in\mathcal{S}\}$
\begin{align*}
\omega^{\prime}[\pi(S)]  &  =\varepsilon,\text{ \ \ \ \ \ \ }\pi(S)\neq\{S\}\\
\omega^{\prime}[S]  &  =0,
\end{align*}
with $0<\varepsilon<m$. Then the sum $\omega+\omega^{\prime}$ is
$h$-increasing and associates a unique minimum cut with each sub-hierarchy
$H(S)$. When $\omega\lbrack\pi^{\ast}(S)]\neq$ $\omega\lbrack\{S\}]$, then the
minimum cut for $\omega+\omega^{\prime}$ is $\pi^{\ast}(S)$, and it is $\{S\}$
itself when $\{S\}$ and $\pi^{\ast}(S)$ have the same $\omega$-energy.%
\begin{align}
\omega\lbrack\pi^{\ast}(S)]  &  \neq\omega\lbrack\{S\}]\text{ \ \ }%
\Rightarrow\text{ \ \ }\pi^{\ast}(S)\text{ best cut}\label{elements7}\\
\omega\lbrack\pi^{\ast}(S)]  &  =\omega\lbrack\{S\}]\text{ \ \ }%
\Rightarrow\text{ \ \ }\{S\}\text{ best cut}\nonumber
\end{align}

\end{proposition}

\begin{proof}
Let us denote by small letters the main inequalities of the proof%
\begin{align*}
(a)\text{\ \ \ \ \ \ \ \ \ \ \ \ }\omega(\pi_{1})  &  \leq\omega(\pi_{2})\\
(b)\text{\ \ \ \ \ \ \ \ \ \ \ \ }\omega(\pi_{1}\sqcup\pi_{0})  &  \leq
\omega(\pi_{2}\sqcup\pi_{0})\\
(c)\text{\ \ \ \ \ \ \ \ \ \ \ \ }\omega(\pi_{1})+\text{\ }\omega^{\prime}%
(\pi_{1})  &  \leq\omega(\pi_{2})+\text{\ }\omega^{\prime}(\pi_{2})\\
(d)\text{\ }\omega(\pi_{1}\sqcup\pi_{0})+\text{\ }\omega^{\prime}(\pi
_{1}\sqcup\pi_{0})  &  \leq\omega(\pi_{2}\sqcup\pi_{0})+\text{\ }%
\omega^{\prime}(\pi_{2}\sqcup\pi_{0})
\end{align*}
We suppose that $(a)\Rightarrow(b)$, and must prove that then $(c)\Rightarrow
(d)$. Let $S_{0}$ be the support of partition $\pi_{0}$. Note firstly that the
equalities $\omega^{\prime}(\pi_{1}\sqcup\pi_{0})=\omega^{\prime}(\pi
_{2}\sqcup\pi_{0})=\varepsilon$ always hold, since $\pi_{1}\sqcup\pi_{0}%
\neq\{S\cup S_{0}\}$ and $\pi_{2}\sqcup\pi_{0}\neq\{S\cup S_{0}\}$.
Distinguish three cases:

1/ $\pi_{1}\neq\{S\}$ and $\pi_{2}\neq\{S\}$. Obviously, $(c)\Rightarrow(a)$
$\Rightarrow(b)\Rightarrow(d)$;

2/ $\pi_{1}\neq\{S\}$ and $\pi_{2}=\{S\}$. Again, $(c)\Rightarrow(a)$
$\Rightarrow(b)\Rightarrow(d)$;

3/ $\pi_{1}=\{S\}$ and $\pi_{2}\neq\{S\}$. Inequality $(c)$ becomes
$\omega(\pi_{1})\leq\omega(\pi_{2})+$\ $\varepsilon$. If $\omega(\pi
_{1})>\omega(\pi_{2})$, then $\omega(\pi_{1})\geq\omega(\pi_{2})+m>\omega
(\pi_{2})+\varepsilon$, which is impossible, hence $\omega(\pi_{1})\leq
\omega(\pi_{2})$. This implies, as previously\bigskip, $(b)$ and then $(c)$.
\end{proof}

\bigskip

Note that the impact of $\omega^{\prime}$ is reduced to the case of equality
$\omega\lbrack\pi^{\ast}(S)]=$ $\omega\lbrack\{S\}]$, and that $\omega
+$\ $\omega^{\prime}$ can be taken arbitrary close to $\omega$, but different
from it. Proposition \ref{unicity} turns out to be a particular case of the
more general, and more useful result

\begin{corollary}
\label{unicity2}Define the additive energy $\omega^{\prime}$ by the relations%
\begin{align*}
\omega^{\prime}[\pi(S)]  &  =\varepsilon,\text{ \ \ \ when\ \ \ \ }\pi
(S)\neq\{S\}\text{ and }\omega\lbrack\pi(S)]\text{\ }\leq\omega_{0}\\
\omega^{\prime}[\pi(S)]  &  =0,\text{ \ \ \ when\ \ \ \ }\pi(S)=\{S\}\text{
and }\omega\lbrack\pi(S)]\text{\ }\leq\omega_{0}\\
\omega^{\prime}[\pi(S)]  &  =\varepsilon,\text{ \ \ \ when\ \ \ \ }%
\pi(S)=\{S\}\text{ and }\omega\lbrack\pi(S)]\text{\ }>\omega_{0}\\
\omega^{\prime}[\pi(S)]  &  =0,\text{ \ \ \ when\ \ \ \ }\pi(S)\neq\{S\}\text{
and }\omega\lbrack\pi(S)]\text{\ }>\omega_{0}.
\end{align*}
Then $\omega+\omega^{\prime}$ is $h-$increasing and always leads to a unique
optimum cut.
\end{corollary}

\begin{proof}
When $\omega\lbrack\pi(S)]$\ $\leq\omega_{0}$, we meet again the previous
proposition \ref{unicity}. When $\omega\lbrack\pi(S)]$\ $>\omega_{0}$, then we
still have the triple distinction of the previous proof, the only difference
being that case 2/ and 3/ are inverted, which achieves the proof.
\end{proof}

\bigskip

Unlike in Proposition \ref{unicity}, in case of equality $\omega\lbrack
\pi^{\ast}(S)]=$ $\omega\lbrack\{S\}]$, the optimum cut is now $\{S\}$ when
$\omega\lbrack\pi(S)]$\ $\leq\omega_{0}$ and $\pi^{\ast}(S)$ when not.
Corollary \ref{unicity2} is used for example for proving that Soille's energy
is $h$-increasing in section \ref{soille} below.

\bigskip

The result (\ref{elements7}) is a top-down property:

\begin{proposition}
\label{top-down}Let $\pi^{\ast}(E)$ be the single minimum cut of a hierarchy
$H$ w.r. to a $h$-increasing energy $\omega$, and let $S$ be a node of $H$. If
$\pi^{\ast}(E)$ meets the sub-hierarchy $H(S)$ of summit $S $, then the
restriction $\pi^{\ast}(S)$ of $\pi^{\ast}(E)$ to $H(S)$ is the single minimum
cut of the sub-hierarchy $H(S).$
\end{proposition}

\begin{proof}
We have to prove Relation (\ref{elements7}). If $\pi^{\ast}(S)=$ $S$, then the
relation is obviously satisfied. Suppose now that the restriction $\pi^{\ast
}(S)$ is different from $S$, and that there exists a cut $\pi_{0}$ of $S$ with
$\omega(\pi_{0})\leq\omega(\pi^{\ast}(S))$. Denote by $\pi^{-}$ the partial
partition obtained when $\pi^{\ast}(S)$ is removed from $\pi^{\ast}(E)$. The
$h$-increasingness then implies that $\omega(\pi_{0}\sqcup\pi^{-})\leq
\omega(S\sqcup\pi^{-})=\omega(\pi^{\ast}(E))$. But by definition of $\pi
^{\ast}(E)$ we also have the reverse inequality, hence $\omega(\pi_{0}%
\sqcup\pi^{-})=\omega(\pi^{\ast}(E))$. Now, this equality contradicts the
Relation (\ref{elements7}) applied to the whole space $E$, so that the
inequality $\omega(\pi_{0})\leq\omega(\pi^{\ast}(S))$ is impossible, and
$\omega(\pi_{0})>\omega(\pi^{\ast}(S))$ for all $\pi_{0}\in\Pi(S)$.
\end{proof}

\subsection{Generation of $h$-increasing energies}

As we saw, the energy $\omega:$ $\mathcal{D}(E)\rightarrow\mathbb{R}^{+\text{
}}$is defined on the family $\mathcal{D}(E)$ of all partial partitions of $E$.
An easy way to obtain a $h$-increasing energy consists in defining it,
firstly, over all sets $S\in\mathcal{P}(E)$, considered as one class partial
partitions $\{S\}$, and then in extending it to partial partitions by some law
of composition. Then, the $h$-increasingness is introduced by the law of
composition, and not by $\omega\lbrack\mathcal{P}(E)]$. The first two modes of
composition which come to mind are, of course, addition and supremum, and
indeed we can state

\begin{proposition}
Let $E$ be a set and $\omega:$ $\mathcal{P}(E)\rightarrow\mathbb{R}^{+\text{
}}$an arbitrary energy defined on $\mathcal{P}(E)$, and let $\pi\in
\mathcal{D}(E)$ be a partial partition of classes $\{S_{i},1\leq i\leq n\}$.
Then the the two extensions of $\omega$ to the partial partitions
$\mathcal{D}(E)$
\[
\omega(\pi)=\vee\{\omega(S_{i}),1\leq i\leq n\}\text{ \ \ and \ \ \ }%
\omega(\pi)=%
{\textstyle\sum}
\{\omega(S_{i}),1\leq i\leq n\}
\]
are $h$-increasing energies.
\end{proposition}

We shall study these two models in sections \ref{additive energies} and
\ref{supremum energies}, when they depend on a parameter leading to multiscale
structures. A number of other laws are compatible with $h$-increasingness.
Instead of the supremum and the sum one could use the infimum, the product,
the difference sup-inf, the quadratic sum, and their combinations. Moreover,
one can make depend $\omega$ on more than one class, on the proximity of the
edges, on another hierarchy, etc..

\subsection{Stucture of the $h$-increasing energies}

We now analyze how different $h$-increasing energies interact on a same
hierarchy. The family $\Omega$ of all mappings $\omega:\mathcal{D\rightarrow
}\mathbb{R}^{+}$ forms a complete lattice where%
\[
\omega\leq\omega^{\prime}\text{ \ \ \ }\Leftrightarrow\text{ \ \ \ }\omega
(\pi)\leq\omega^{\prime}(\pi)\text{ \ \ for all }\pi\in\mathcal{D}%
\]
and whose extrema are $\omega(\pi)=0$ and \ $\omega(\pi)=+\infty$. What can be
said about the sub class of$\ \Omega^{\prime}$ $\subseteq\Omega$ of the
$h$-increasing energies? The class $\Omega^{\prime}$ is obviously closed under
addition and multiplication by positive scalars, i.e.%
\begin{equation}
\{\omega^{j}\}\subseteq\Omega^{\prime},\lambda_{j}\geq0\text{ \ \ }%
\Rightarrow\text{ \ \ }%
{\textstyle\sum}
\lambda_{j}\omega^{j}\in\Omega^{\prime} \label{hier 17}%
\end{equation}

Consider\ now a finite family $\{\omega_{i},i\in I\}$ in $\ \Omega^{\prime}$
such that, for \ $\pi_{0},\pi_{1},\pi_{2}\in\mathcal{D}(E),\ \pi_{0}$ disjoint
of $\pi_{1}$ and $\pi_{2}$, we have
\begin{equation}
\omega_{i}(\pi_{1})\leq\omega_{j}(\pi_{2})\text{ \ \ \ \ }\Rightarrow\text{
\ \ }\omega_{i}(\pi_{1}\sqcup\pi_{0})\leq\ \omega_{j}(\pi_{2}\sqcup\pi
_{0}),\text{ \ \ }\pi_{0},\pi_{1},\pi_{2}\in\mathcal{D}\text{\ \ }
\label{hier 13}%
\end{equation}
a relation that generalizes the $h$-increasingness (\ref{croiss_hierarch}).

\begin{proposition}
The family of those energies that satisfy implication (\ref{hier 13}) is a
finite sub-lattice $\Omega^{\prime}$of $\Omega$, and for any family
$\{\omega_{i},i\in I\}$ in $\Omega^{\prime}$ we have%
\begin{align}
(\wedge\omega_{i})(\pi)  &  \leq(\wedge\omega_{i})(\pi^{\prime})\text{
\ \ \ }\Rightarrow\text{ \ \ }(\wedge\omega_{i})(\pi\sqcup\pi_{0})\leq
(\wedge\omega_{i})(\pi^{\prime}\sqcup\pi_{0})\label{hier 14}\\
(\vee\omega_{i})(\pi)  &  \leq(\vee\omega_{i})(\pi^{\prime})\text{
\ \ \ }\Rightarrow\text{ \ \ }(\vee\omega_{i})(\pi\sqcup\pi_{0})\leq
(\vee\omega_{i})(\pi^{\prime}\sqcup\pi_{0})
\end{align}

\end{proposition}

\begin{proof}
Consider a family $\{\omega_{i},i\in I\}$ which satisfy the $h$-increasingness
(\ref{hier 13}). Suppose that $(\wedge\omega_{i})(\pi)\leq(\wedge\omega
_{i})(\pi^{\prime})$, and let $i_{0}$ be the parameter of the smallest of the
$\omega_{i}(\pi)$. Then we have $\omega_{i_{0}}(\pi)\leq\omega_{i}(\pi
^{\prime}),i\in I$, hence, from Rel.(\ref{hier 13}), $\omega_{i_{0}}(\pi
\sqcup\pi_{0})\leq\omega_{i}(\pi^{\prime}\sqcup\pi_{0})$, so that
$\omega_{i_{0}}(\pi\sqcup\pi_{0})\leq(\wedge\omega_{i})(\pi^{\prime}\sqcup
\pi_{0})$, and finally $(\wedge\omega_{i})(\pi\sqcup\pi_{0})\leq(\wedge
\omega_{i})(\pi^{\prime}\sqcup\pi_{0})$. Energy $(\wedge\omega_{i})$ is thus
$h$-increasing. Same proof for the supremum.
\end{proof}

\bigskip

Note that the infimum cut $\pi_{\wedge\omega_{i}}$ related to energy
$\wedge\omega_{i}$ \textit{is not} the infimum $\wedge\pi_{\omega_{i}}^{\ast}$
of the minimum cuts generated by the $\omega_{i}$. It has to be computed
directly (dual statement for the supremum).

\section{climbing energies}

The usual energies are often given by finite sequences $\{\omega^{j},1\leq
j\leq p\}$ that depend on a positive index, or parameter, $j$. Therefore, the
processing of hierarchy $H$ results in a sequence of $p$ optimum cuts
$\pi^{j\ast}$, of labels $1\leq j\leq p$. A priori, the $\pi^{j\ast}$ are not
ordered, but is they were, i.e. if
\[
j\leq k\ \ \ \Rightarrow\ \ \ \pi^{j\ast}\leq\pi^{k\ast},\ \ \ \ \ j,k\in
J,\qquad
\]
\ \ \ then we should obtain a nice progressive simplification of the optima.
We now seek the conditions that permit such an increasingness.

\bigskip

Consider a finite family of energies $\{\omega^{j},1\leq j\leq p\}$ on all
partial partitions $\mathcal{D}(E)$ of set $E$ , and apply these energies to
the partial partitions $\widetilde{\Pi}(H)$ of hierarchy $H$ (Relation
(\ref{guigues 3})). The family $\{\omega^{j}\}$, not totally arbitrary, is
supposed to satisfy the following condition of \textit{scale increasingness}%
:\footnote{Scale increasingness is called by L.Guigues "partial causality" in
\cite{GUI03}, p.161, where it appears as a consequence and \ not as a starting
point.}

\begin{definition}
\label{multi-scale} A family of energies $\{\omega^{j},1\leq j\leq p\}$ on
$\mathcal{D}(E)$ is said to be \emph{scale increasing} when for $j\leq k$,
each support $S\in\mathcal{S}$ and each partition $\pi\in\Pi(S)$, we have that%
\begin{equation}
j\leq k\ \ \text{and \ }\omega^{j}(S)\leq\omega^{j}(\pi)\Rightarrow\omega
^{k}(S)\leq\omega^{k}(\pi)\text{, \ \ \ \ \ \ \ }S\in\mathcal{P}(E)\text{.}
\label{hier 1}%
\end{equation}

\end{definition}

In case of a hierarchy $H$, relation (\ref{hier 1}) means that, if $S$ is a
minimum cut w.r. to energy $\omega^{j}$ for a partial hierarchy $\Pi(S)$, then
$S$ remains a minimum cut of $\Pi(S)$ for all energies $\omega^{k}$, $k\geq$
$j$. As $j$ increases, the $\omega^{j}$'s preserve the sense of energetic
differences between the nodes of hierarchy $H$ and their partial partitions.
In particular, all energies of the type $\omega^{j}=j\omega$ are scale increasing.

Axiom (\ref{hier 1}) compares two energies at the same level of $H$, whereas
axiom (\ref{croiss_hierarch}) allows us to compare a same energy at two
different levels. Therefore, the most powerful energies should to be those
which combine scale and $h$-increasingness, i.e.

\begin{definition}
\label{multiscale}We call \emph{climbing energy} any family $\{\omega
^{j},1\leq j\leq p\}$ of energies over $\widetilde{\Pi}$ which satisfies the
three following axioms, valid for $\omega^{j},1\leq j\leq p$ and for all
$\pi\in\Pi(S)$, $S\in\mathcal{S}$

\begin{itemize}
\item i) \emph{$h$-increasingness, }i.e. relation(\ref{croiss_hierarch}%
)\emph{: }%
\[
\omega^{j}(\pi_{1})\leq\omega^{j}(\pi_{2})\ \ \Rightarrow\ \ \omega^{j}%
[(\pi_{1}\sqcup\pi_{0}]\leq\ \omega^{j}[(\pi_{2}\sqcup\pi_{0}],\text{
\ \ \ \ }\pi_{1},\pi_{2}\in H(S)
\]

\item ii) \emph{single minimum cutting:}
\[
\text{either }\omega^{j}[\pi^{\ast}(S)]<\omega^{j}(\pi),\ \ \text{or }%
\ \pi^{\ast}(S)=S,\
\]

\item iii) \emph{scale increasingness, }i.e.\ relation(\ref{hier 1})\emph{\ :}%
\[
j\leq k\ \ \text{and }\omega^{j}(S)\leq\omega^{j}(\pi)\Rightarrow\omega
^{k}(S)\leq\omega^{k}(\pi)\text{, \ \ \ \ \ \ \ }\pi\in H(S)\text{.}%
\]

\end{itemize}
\end{definition}

Under these three assumptions, the climbing energies satisfy the very nice
property to order the minimum cuts with respect to the parameter $j$, namely:

\begin{theorem}
\label{pyram_croit}Let $\{\omega^{j}$, $1\leq j\leq p\}$ be a family of
energies, and let $\pi^{j\ast}$ (resp. $\pi^{k\ast}$) be the minimum cut of
hierarchy $H$ according to the energy $\omega^{j}$ (resp. $\omega^{k}$). The
family $\{\pi^{j\ast}$,$1\leq j\leq p\}$ of the minimum cuts generates a
unique hierarchy $H^{\ast}$ of partitions, i.e.
\begin{equation}
j\leq k\text{ \ \ \ }\Rightarrow\text{ \ \ \ }\pi^{j\ast}\leq\pi^{k\ast
}\text{, \ \ \ \ \ }1\leq j\leq k\leq p \label{hier 6}%
\end{equation}
if and only if the family $\{\omega^{j}\}$ is a climbing energy.
\end{theorem}

\begin{proof}
Assume that axiom $iii)$ of a climbing energy is satisfied, and denote by
$S^{j}$ and $S^{k}$ the two classes of $\pi^{j\ast}$ and $\pi^{k\ast}$ at a
given point $x$. According to Rel.(\ref{hierarchy 2}), we must have either
$S^{j}$ $\subseteq$ $S^{k}$ or $S^{k}\subset S^{j}$. We will prove that the
second inclusion is impossible. Suppose that class $S^{k}\subset S^{j}$. Then
the restriction of the minimum cut $\pi^{k\ast}$ to $S^{j}$ generates a cut,
$\pi_{0}$ say, of $S^{j}$. According to proposition \ref{top-down}, which
involves axioms $i)$ and $ii)$ of a climbing energy, the restriction $\pi_{0}$
is in turn minimum for the energy $\omega^{k}$ over $\Pi(S^{j})$, i.e.
$\omega^{k}(\pi_{0})<\omega^{k}(S^{j})$. This implies, by scale
increasingness, that $\omega^{j}(\pi_{0})<\omega^{j}(S_{j})$ (here Relation
(\ref{hier 1}) has been red from right to left). But this inequality
contradicts the fact that $S^{j}$ is a minimum cut for its own hierarchy
$H(S^{j})$. Therefore the inclusion $S^{k}\subset S^{j}$ is rejected, and the
alternative inclusion $S^{j}$ $\subseteq$ $S^{k}$ is satisfied whatever $x\in
E$, which results in $\pi^{j\ast}\leq\pi^{k\ast}$. Moreover, because of the
single minimum cutting axiom, each $\pi^{j\ast}$ being unique, so does the
whole hierarchy $H^{\ast}$.

The "only if" statement will be proved by means of a counter-example. The
notation is the same as in system (\ref{Q}) but we add a term for the areas
and we replace the length $\partial T_{u}$ of the frontier of each $T_{u}$ by
the length $\partial_{hor}T_{u}$ (resp. $\partial_{vert}T_{u}$) of the
horizontal (resp. vertical) projection of the said frontiers if $j\leq j_{0},
$ (resp. if\ $j>j_{0}$).
\begin{align}
\omega^{j}(\pi(S))  &  =\omega(T_{1}\sqcup..T_{u}..\sqcup T_{q})=4\sum
\limits_{1\leq u\leq q}(areaT_{u})^{2}+2\sum\limits_{1\leq u\leq q}%
(\partial_{hor}T_{u}),\text{\ \ when \ \ \ }j\leq j_{0}\label{QQ}\\
\omega^{j}(\pi(S))  &  =\omega(T_{1}\sqcup..T_{u}..\sqcup T_{q})=4\sum
\limits_{1\leq u\leq q}(areaT_{u})^{2}+2\sum\limits_{1\leq u\leq q}%
(\partial_{vert}T_{u}),\text{ \ when \ \ \ }j>j_{0},\nonumber
\end{align}
(border lengths being divided by 2). For each value of $j$, the energy
$\omega_{1}^{j}$ is $h$-increasing, and one can apply proposition
\ref{unicity} for ensuring unicity. However, if $j\leq j_{0}$ then
$\omega(T_{1}\sqcup T_{2})=6,\omega(T_{1}^{\prime}\sqcup T_{2}^{\prime
})=4,\omega(S)=$\ $\omega(S^{\prime})=6,\omega(E)=19$.\ Therefore the minimum
cut is the partition $T_{1}^{\prime}\sqcup T_{2}^{\prime}\sqcup S$, depicted
in Figure \ref{counter_ex1} d). By symmetry, if $j>j_{0}$ then\ the minimum
cut becomes $T_{1}\sqcup T_{2}\sqcup S^{\prime}$, of Figure \ref{counter_ex1}
c). Now these two cuts are not comparable, which contradicts implication
(\ref{hier 6}).
\end{proof}

\bigskip

Relation (\ref{hier 6}) has been established by L.\ Guigues in his Phd thesis
\cite{GUI03} for affine and separable energies, called by him climbing
energies. However, the core of the assumption (\ref{hier 1}) concerns the
propagation of energy through the scales $(1...p)$, rather than affinity or
linearity, and allows non additive laws (see Section \ref{supremum energies}).
In addition, the new climbing axioms \ref{multiscale} lead to the algorithms
of the next section, much simpler than that of \cite{GUI03}.

\bigskip

The scale sequence has been supposed finite, i.e. with $1\leq j\leq p$. We
could replace $j$ by a positive parameter $\lambda\in$ $\mathbb{R}_{+}$. But
the induction method for finding the optimal cut requires a finite hierarchy
$H$. Therefore the number of the scale parameters actually involved in the
processing of $H$ will always be finite.

\section{Implementation via a pedagogical example}%

\begin{figure}
[ptb]
\begin{center}
\ifcase\msipdfoutput
\includegraphics[
height=3.7421in,
width=5.2477in
]%
{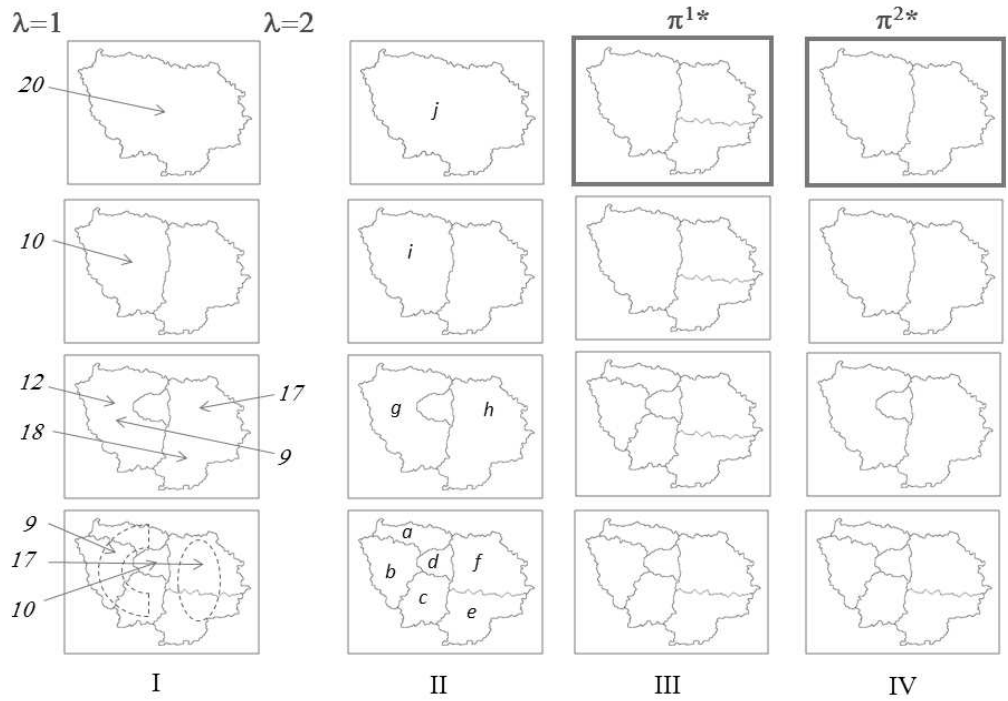}%
\else
\includegraphics[
height=3.7421in,
width=5.2477in
]%
{C:/SWtexts/2012/climbing_on_pyramids_2012/graphics/example__6.pdf}%
\fi
\caption{Column I: initial hierarchy and associated energies (left, for
$\lambda=1$; right, the changes for $\lambda=2$); column II: reading order;
columns III and IV: progressive extraction of the minimum cuts for $\lambda=1$
and $\lambda=2$, the two final minuma have bold frames. }%
\label{example}%
\end{center}
\end{figure}

We now describe step by step two algorithms for generating a hierarchy of
minimum cuts. The input hierarchy $H$ comprises the four levels and two
energies $\{\omega_{1},\omega_{2}\}$ depicted in Figure \ref{example}. The
energies of the classes are given along column I, left for $\lambda=1$, and
right for the energies that are different for $\lambda=2.$They are composed by
supremum (the energies of the three partial partitions $\{a,b,c\}$, and
$\{e,f\}$ are not decomposed). If class $S$ is temporary minimum for
$\omega_{\lambda}$ then is remains temporary minimum for all $\omega_{\mu},$
$\mu\geq\lambda$. The \textit{scale of apparition }$\lambda^{+}(S)=\inf
\{\lambda\mid S$ temporary minimum for $\omega_{\lambda}\}$ is the smallest
$\lambda$ for which $S$ is a temporary minimum. A class $S$ which is covered
by a temporary minimum class $Y$ at scale $\mu$ remains covered by temporary
minima at all scales $\nu\geq\mu$. The \textit{scale of removal} $\lambda
^{-}(S)=\min\{\lambda^{+}(Y),Y\in H,S\subseteq Y\}$ is the smallest $\mu$ for
which $S$ is covered. Therefore, if $\lambda^{+}(S)<\lambda^{-}(S)$, then
class $S\in H$ belongs to the minimum cuts $\pi^{\lambda\ast}$ for all
$\lambda$ of the \textit{interval of persistence} $[\lambda^{+}(S),\lambda
^{-}(S)]$. If $\lambda^{-}(S)\leq\lambda^{+}(S)$, then $S$, non-persistent,
never appears as a class of a minimum cut. This happens to class $g$, for
which $\lambda^{-}(g)=1$ and $\lambda^{+}(g)=2$.

When the two bounds $\lambda^{+}(S)$ and $\lambda^{-}(S)$ are known for all
classes $S\in H$ then the hierarchy of the minimum cuts $\pi^{\lambda\ast}$ is
completely determined.They are calculated in two passes as follows

\qquad1-take the nodes following their labels (ascending pass). For node $S$,
calculate the two energies $\omega_{\lambda}$ of $S$ and of its sons $\pi(S)$,
for $\lambda=1,2$. Stop at the first $\lambda$ such that $\omega_{\lambda
}(S)\leq\omega_{\lambda}(\pi(S))$. This value is nothing but $\lambda^{+}(S)$.
Continue for all nodes until he top of the hierarchy. This pass provides the
$\lambda^{+}(S)$ values for all $S\in H$. In Figure \ref{example} the sequence
of the two energies is climbing: when a class is temporary minimum for
$\lambda=1$, it also does for $\lambda=2$. Two changes occur from $\lambda=1$
to $\lambda=2$ : $g$ and $h$ have now the same energies as $\{a,b,c\}$ and
$\{e,f\}$respectively.

\qquad2-The second pass progresses top down. Each class $Y$ is compared to its
sons $Z_{1},.Z_{i}$, and one allocates to each son $Z_{i}$ the new value
$\min\{\lambda^{-}(Y),\lambda^{-}(Z_{i})\}$. At the end of the scan, all
values $\lambda^{-}(S)$ are known.

\section{Additive energies\label{additive energies}}

The additive mode was introduced and studied by L.Guigues under the name of
\textit{separable energies }\cite{GUI03}. All classes $S$ of $\mathcal{S}$ are
supposed to be connected. Denote by $\{T_{u},1\leq u\leq q\}$ the $q$ sons
which partition the node $S$, i.e. $\pi(S)=T_{1}\sqcup..T_{u}..\sqcup T_{q}$.
Provide the simply connected sets of $\mathcal{P}(E)$ with energy $\omega$,
and extend it from $\mathcal{P}(E)$ to the set $\mathcal{D}(E)$ of all partial
partitions by using the sums%
\begin{equation}
\omega(S)=\omega(T_{1}\sqcup..T_{u}..\sqcup T_{q})=\sum\limits_{1}^{q}%
\omega(T_{u}). \label{hier 10}%
\end{equation}
All separable energies $\omega$ are clearly $h$-increasing on any hierarchy,
since one can decompose the second member of implication
(\ref{croiss_hierarch}) into $\omega(\pi_{1}\sqcup\pi_{0})=\ \omega(\pi
_{1})+\omega(\pi_{0})$, and $\ \omega(\pi_{2}\sqcup\pi_{0})=\ \omega(\pi
_{2})+\omega(\pi_{0})$. However, they do not always lend themselves to
multiscale structures, and a supplementary assumption of affinity has to be
added \cite{GUI03}, by putting
\begin{equation}
\omega^{j}(S)=\omega_{\mu}(S)+\lambda^{j}\omega_{\partial}(S)\text{
\ \ \ \ \ \ \ }S\in\mathcal{S} \label{affine_energy}%
\end{equation}
where $\omega_{\mu}$ is a goodness-of-fit term, and $\omega_{\partial}$ a
regularization one, and $\lambda^{j}\geq0$ an increasing function of $j$. The
term $\omega_{\mu}$ is associated with the interior of $A$ and $\omega
_{\partial}$ with its boundary. In such an affine energy, the function
$\omega_{\partial}$ must be $c$-additive on the boundary arcs $F$, in the
sense of integral geometry, i.e. one must have
\begin{equation}
\omega_{\partial}(F_{1}\cup F_{2})+\omega_{\partial}(F_{1}\cap F_{2}%
)=\omega_{\partial}(F_{1})+\omega_{\partial}(F_{2}) \label{hier 11}%
\end{equation}
since most of the arcs $F_{1},F_{2},...F_{i}$ of the boundaries are shared
between two adjacent classes. The passage "s\textit{et}$\rightarrow
$\textit{partial partition}" can then be obtain by the summation
(\ref{hier 10}). One classically take for $\omega_{\partial}$ the arc length
function (e.g. in Rel.(\ref{hier 8}) ), but it is not the only choice. Below,
one of the examples by Salemenbier and Garrido about thumbnails uses
$\omega_{\partial}(S)=1$. One can also think about another $\omega_{\partial
}(S)$, which reflects the convexity of $A$.

\bigskip

The axiom (\ref{hier 1}) of scale increasingness involves only increments of
the energy $\omega$, which suggests the following obvious consequence:

\begin{proposition}
\label{scale_add}If a family $\{\omega^{j},1\leq j\leq p\}$ of energies is
scale increasing, then any family $\{\omega^{j}+\omega_{0},1\leq j\leq p\}$,
where $\omega_{0}$ is an arbitrary energy over $\widetilde{\Pi}$ which does
not depends on $j$, is in turn scale increasing.
\end{proposition}

\subsection{Additive energy and convexity}

Consider, indeed, in $\mathbb{R}^{2}$ a compact connected set $X$ without
holes, and let $d\alpha$ be the elementary rotation of its outward normal
along the element $du$ of the frontier $\partial X$. As the radius of
curvature $R$ equals $du/d\alpha$, and as the total rotation of the normal
around $\partial X$ equals $\pi$, we have%
\begin{equation}
2\pi=\int\limits_{R\geq0}\frac{du}{R(u)}+\int\limits_{R>0}\frac{du}{\left\vert
R(u)\right\vert }. \label{hier 12}%
\end{equation}
When dealing with partitions, the distinction between outward and inward
vanishes, but the parameter%
\[
n(X)=\frac{1}{2\pi}\int\limits_{\partial X}\frac{du}{\left\vert
R(u)\right\vert }%
\]
still makes sense. It reaches its minimum $1$ when set $X$ is convex, and
increases with the degree of concavity. For a starfish with 5 pesudo-podes, it
values around 5. Now $n(X)$ is $c$-additive for the open parts of contours,
therefore it can participate as a supplementary term in an additive
energy.\ In digital implementation, the angles between contour arcs must be
treated separately (since $c$-additivity applies on the open parts).

\subsection{ Mumford and Shah energy}

Let $\pi(S)$ be the partition of a summit $S$ into its $q$ sons $\{T_{u},1\leq
u\leq q\}$ i.e. $\pi(S)=T_{1}\sqcup..T_{u}..\sqcup T_{q}$. The classical
Mumford and Shah energy $\omega^{j}$ on $\pi(S)$ and w.r.t. a function $f$
comprises two terms \cite{MUM88}. The first one sums up the quadratic
differences between $f$ and its average $m(T_{u})$ in the various $T_{u}$, and
a second term weights by $\lambda^{j}$ the lengths $\partial T^{i}$ of the
frontiers of all $T_{u}$, i.e.
\begin{equation}
\omega^{j}(\pi(S))=\sum\limits_{1\leq u\leq q}\int_{x\in T^{u}}\parallel
f(x)-m(T_{u})\parallel^{2}+\lambda^{j}\sum\limits_{1\leq u\leq q}(\partial
T_{u})=\omega_{\mu}(\pi)+\lambda^{j}\omega_{\partial}(\pi) \label{hier 8}%
\end{equation}
where the weight $\lambda^{j}$ is a numerical increasing function of the level
number\ $j$. Both increasingness relations (\ref{croiss_hierarch}) and
(\ref{hier 1}) are satisfied and Theorem \ref{pyram_croit} applies to the
additive energy (\ref{hier 8}).

\subsubsection{Additive energies and Lagrangian}

The example of additive energy that we now develop is an extension of the
creation of thumbnails by Ph. Salembier and L. Garrido \cite{Sal20}
\cite{SAL11}, itself based on Equation (\ref{hier 8}). We aim to generate "the
best" simplified version of the image $f$ of Fig. \ref{talbot}, in its color
version of components ($r;g;b$), for the compression rate $\rho\simeq25$. The
bit depth of $f$ is $24$ and its size is $320\times416$ pixels. The hierarchy
$H$ of \ $f$ is that depicted, via its saliency map, in Fig. \ref{talbot}c. It
has been obtained from the luminance $l=(r+g+b)/3$. In each class $S$ of $H$,
the reduction consists in replacing the function $f$ by its mean $m(l)$ in
$S$. The quality of this approximation is estimated by the $L_{2}$ norm, i.e.%
\begin{equation}
\omega_{\mu}(S)=\sum\limits_{x\in S}\parallel l(x)-m(S)\parallel^{2}
\label{hier 16}%
\end{equation}

If the coding cost for a frontier element is $c$, that of the whole class $S$
becomes%
\begin{equation}
\omega_{\partial}(S)=24+\frac{c}{2}\mid\partial S\mid\label{hier 15}%
\end{equation}
with $24$ bits for $m(S)$, and we take $c=2$ for elementary cost. The total
energy of a cut $\pi$ is thus written $\omega^{j}(\pi)=$ $\omega_{\mu}%
(\pi)+\lambda^{j}\omega_{\partial}(\pi)$. By applying Lagrange's theorem, we
observe that the problem of finding the minimum of $\omega_{\mu}(\pi)$ under
the constraint $\omega_{\partial}(\pi)\simeq25$ implies that the Lagrangian
$\omega_{\mu}(\pi)+\lambda^{j}\omega_{\partial}(\pi)$\ is a minimum. In this
equation, the level $j$ is unknown, but as the $\omega^{j}$ are multiscale, we
can easily calculate the sequence $\{\pi^{j\ast}$, $1\leq j\leq p\}$ of the
minimum cuts. Moreover, the term $\omega_{\partial}(\pi)$ itself turns out to
be a decreasing function of $j $ and $\lambda^{j}$, so that the solution of
the Lagrangian is the $\pi^{j\ast}$ whose term $\omega_{\partial}(\pi^{j\ast
})$ is the greatest one smaller than $k$. It is depicted in
Fig.\ref{talbot_cuts} a (in a black and white version).%

\begin{figure}
[ptb]
\begin{center}
\ifcase\msipdfoutput
\includegraphics[
height=2.3237in,
width=5.1871in
]%
{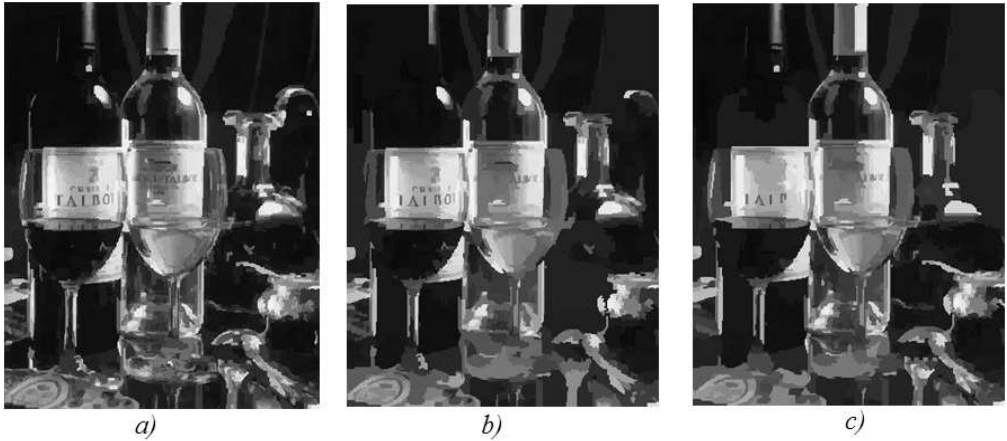}%
\else
\includegraphics[
height=2.3237in,
width=5.1871in
]%
{C:/SWtexts/2012/climbing_on_pyramids_2012/graphics/Talbot_cuts__7.pdf}%
\fi
\caption{ Minimum cuts of the hierarchy Fig. \ref{talbot}c, for two
compression rates of 25 \ and 55 applied to the luminance ( a) and b)
respectively); c) for a compression rate of 55\ applied to the chrominance. }%
\label{talbot_cuts}%
\end{center}
\end{figure}

\bigskip

Classically one reaches the Lagrangian\ minimum value by means of a system of
partial derivatives. Now, remarkably, the present approach replaces the of
computation of derivatives by a climbing. Moreover, we have under hand, at
once, all best cuts for all compression rates. If we take $\rho\simeq55$ for
example, we find the image of Fig.\ref{talbot_cuts} b, whose partition is
located at a higher level in the same pyramid of the best cuts. L.\ Guigues
was the first to point out\ this nice property \cite{GUI03}.

\bigskip

There is no particular reason to choose the same luminance $l$ for generating
the pyramid and, later, as the quantity to involve in the quality
(\ref{hier 16}) to minimize. In the $RGB$ space, a colour vector
$\overrightarrow{x}(r;g;b)$ can be decomposed in two orthogonal projections

\textit{i}) on the grey axis, namely$\overrightarrow{l}$of components
$(l/3;l/3;l/3)$,

\textit{ii}) and on the chromatic plane orthogonal to the grey axis at the
origin, namely $\overrightarrow{c}$ of components$(3/\sqrt{2}%
)(2r-g-b;2g--b-r;2b-r-g)$.

We have $\overrightarrow{x}=\overrightarrow{l}+$ $\overrightarrow{c}$. The
optimization is obtained by replacing the luminance $l(x)$ in (\ref{hier 16})
by the module $\mid\overrightarrow{c}(x)\mid$ of the chrominance at point $x$.
We now find for best cut the segmentation depicted in Fig.\ref{talbot_cuts} c,
where, for the same compression rate $\rho\simeq55$, color details are better
rendered (e.g. right bottom), but black and white parts are worse (e.g. the
letters on the labels).

\section{Energies composed by supremum\label{supremum energies}}

We now us go back to Rel.(\ref{hier 10}), and replace the sum by a supremum.
It gives:%
\[
\omega(S)=(T_{1}\sqcup..T_{u}..\sqcup T_{q})=\bigvee\limits_{1}^{q}%
\omega(T_{u}).
\]
This time the second member of implication (\ref{croiss_hierarch}) is written
$\omega\lbrack(\pi_{1}\sqcup\pi_{0}]=\ \omega(\pi_{1})\vee\omega(\pi_{0})$,
and $\ \omega\lbrack(\pi_{2}\sqcup\pi_{0}]=\ \omega(\pi_{2})\vee\omega(\pi
_{0})$, so that $\omega$ is $h$-increasing. The property remain true when
$\vee$ is replaced by $\wedge$.

\subsection{Connective segmentation under constraint \label{soille}}%

\begin{figure}
[ptb]
\begin{center}
\ifcase\msipdfoutput
\includegraphics[
height=2.4465in,
width=3.8761in
]%
{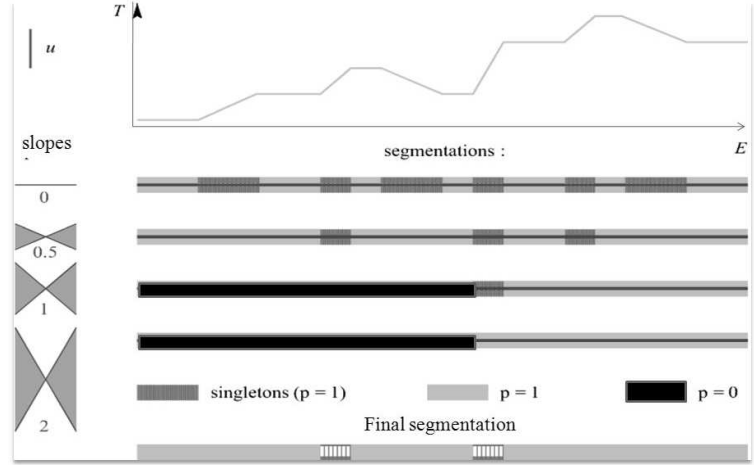}%
\else
\includegraphics[
height=2.4465in,
width=3.8761in
]%
{C:/SWtexts/2012/climbing_on_pyramids_2012/graphics/ronse__8.pdf}%
\fi
\caption{\ Function $f$ is segmented by smooth path connections of increasing
slopes. Then one takes at each point the largest class $S$ with $\Delta
_{f}(S)\leq c$ (from C.\ Ronse \cite{RON09b}).}%
\label{ronse}%
\end{center}
\end{figure}

This method was proposed by P. Soille and J. Grazzini in \cite{SOI08a} and
\cite{SOI09a} with several variants; it is re-formulated by C.\ Ronse in a
more general framework in \cite{RON09b}.\ Start from a hierarchy $H$ and a
numerical function $f$. Define the energy $\omega^{j}$ for class $S$ by
\begin{align*}
\omega^{j}(S)  &  =0\text{ \ \ \ when \ }\sup\{f(x),x\in S\}-\inf\{f(x),x\in
S\}\leq c^{j}\\
\omega^{j}(S)  &  =1\text{ \ \ \ when not,}%
\end{align*}
where $c^{j}$ is a given bound, and extend to partitions by $\vee
$-composition. The class at point $x$\ of the largest partition of minimum
energy is given by the largest $S\in\mathcal{S}$, that contains $x$, and such
that the amplitude of variation of $f$ inside $S$ be $\leq c^{j}$. When the
energy $\omega^{j}\ $of a father equals that of its sons, one keeps the father
when $\omega^{j}=0$, and the sons when not. As bound $c^{j}$ increases, the
$\{\omega^{j}\}$ form a climbing energy, previously referred to in corollary 
\cite{unicity2} and depicted by the pedagogical example of Figure \ref{ronse} .

\subsection{Lasso}

This algorithm, due to F. Zanoguera et Al. \cite{ZAN99}, appears also in
\cite{MEY09a}. An initial image has been segmented into $\alpha$-flat zones,
which generates a hierarchy as the slope $\alpha$ increases . The optimization
consists then in drawing manually a closed curve around the object to segment.
If $A$ designates the inside of this lasso, then we take the following
function
\begin{align}
\omega(S)  &  =0\text{ \ \ \ when \ }S\subseteq A\text{ ;\ \ }S\in
\mathcal{S}\label{hier 9}\\
\omega(S)  &  =1\text{ \ \ \ when not,}\nonumber
\end{align}
for energy, and we go from classes to partitions by $\vee$-composition of the
energies. The largest cut that minimizes $\omega$ is depicted in Figure
\ref{girasol}c. We see that the resulting contour follows the edges of the
petals. Indeed, a segmented class can jump over this high gradient only for
$\alpha$ large enough, and then this class is rejected because it spreads out
beyond the limit of the lasso.%

\begin{figure}
[ptb]
\begin{center}
\ifcase\msipdfoutput
\includegraphics[
height=1.644in,
width=3.2794in
]%
{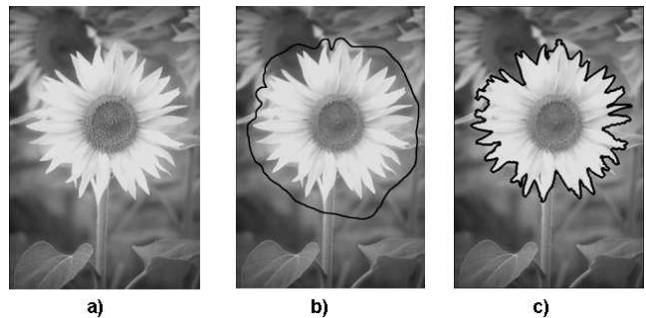}%
\else
\includegraphics[
height=1.644in,
width=3.2794in
]%
{C:/SWtexts/2012/climbing_on_pyramids_2012/graphics/girasol__9.pdf}%
\fi
\caption{a) Initial image; b) manual lasso; c) contour of the union of the
classes inside the lasso.}%
\label{girasol}%
\end{center}
\end{figure}

By taking a series $\{A^{j},1\leq j\leq p\}$ of increasing lassos $A^{j}$ we
make climbing the energy $\omega$ of Relation (\ref{hier 9}), and Theorem
\ref{pyram_croit} applies. Unlike the previous case, where the $\lambda^{j}$
are scalars numbers, the multiscale parametrization is now given by a family
of increasing sets.

\section{Conclusion}

Other examples can be given, concerning colour imagery in particular
\cite{ANG06a}. At this stage, the main goal is to extend the above approach to
vector data, and more generally to GIS type data.

\end{document}